\documentclass[final,3p,twocolumn]{elsarticle}

\usepackage{amsmath,amssymb,amsfonts,amstext,graphics,graphicx,subfigure}
\usepackage{setspace}
\usepackage[colorlinks]{hyperref}

\usepackage{color}

\usepackage{mathbbol}

\newtheorem{theorem}{Theorem}

\newtheorem{definition}[theorem]{Definition}

\newtheorem{remark}[theorem]{Remark}

\newenvironment{proof}[1][Proof]{\noindent\textbf{#1.} }{\ \rule{0.5em}{0.5em}}

\def\XXint#1#2#3{{\setbox0=\hbox{$#1{#2#3}{\int}$ }
\vcenter{\hbox{$#2#3$ }}\kern-.5\wd0}}

\def\calv{\mathcal{V}}

\def\calz{\mathcal{Z}}

\def\R{\mathbb{R}}

\def\del{\partial}

\def\bq{\begin{equation}}
\def\eq{\end{equation}}
\def\bqy{\begin{eqnarray}}
\def\eqy{\end{eqnarray}}

\def\bal#1\eal{\begin{align}#1\end{align}}

\def\quadd{\quad\quad}

\def\al{\alpha}

\def\de{\delta}

\def\La{\Lambda}

\def\om{\omega}

\def\si{\sigma}

\def\bfd{\mathbf{d}}

\newcommand{\KN}{\mathbin{\bigcirc\mspace{-15mu}\wedge\mspace{3mu}}}

\def\p{\partial}

\def\andq{\quad\mathrm{and}\quad}
\def\andqq{\quadd\mathrm{and}\quadd}

\def\nn{\nonumber}

\begin{document}

\begin{frontmatter}

 \title{Metriplectic dynamical systems on contact manifolds}

 \author[inst1]{Philip J Morrison}
 \affiliation[inst1]{organization={Department of Physics and Institute for Fusion Studies, University of Texas at Austin},
            addressline={}, 
            city={Austin},
            postcode={78712}, 
            state={TX},
            country={USA}}
 \ead{morrison@physics.utexas.edu}
 
\author[inst2]{Yong-Geun Oh} 
   \affiliation[inst2]{organization={Center for Geometry and Physics, Institute for Basic Science (IBS) Nam-gu}, 
           addressline={77 Cheongam-ro}, 
         city={Pohang-si},
           state={Gyeongsangbuk-do},
            country={Korea}
}
\ead{yongoh1@postech.ac.kr}
 
%
%


\begin{abstract} 
Flows on symplectic, Poisson, contact, and metriplectic manifolds are reviewed in order to describe our main result, which  is to associate a natural \emph{metriplectic} dynamical system on the general one-jet bundle $J^1N = T^*N \times \R$, which is at once a (trivial) Poisson manifold and a contact manifold. Unlike the standard contact Hamiltonian system, our
metriplectic system is thermodynamically consistent in that 
$$
\dot H = 0 \quad \mathrm{and}\quad \dot S \geq 0
$$
under the flow. Here $H$ is the Hamiltonian, while $S$ is the \emph{entropy} function
which is nothing but the $\R$ coordinate function of $J^1N$.  As an example we derive the Duffing equation (autonomous and  nonautonomous versions)  either as a contact Hamiltonian system or as a metriplectic system.  
We show that for both systems  the Duffing equation is a subsystem of three dimensional systems that contain a thermodynamic component, {a form that facilitates asymptotic stability analysis of the relevant equilibrium state.}
\end{abstract}



\end{frontmatter}

\section{Introduction}
\label{sec:intro}

It is widely known that Gibbs  \cite{gibbs}  understood the contact structure of classical thermodynamics, and that this geometric structure underlies the different potential  representations of thermodynamics associated by Legendre transforms. However, classical thermodynamics is not dynamics,  so actual temporal  evolution is not part of the theory. So, in a sense, this paper is about dynamical thermodynamics;  in particular, it is about how contact structure appears in this setting and how it relates to another geometric structure, metriplectic dynamics \cite{pjm86}, an inclusive  dynamical systems theory for thermodynamically consistent systems.  

By definition a  dynamical system  is thermodynamically consistent if it satisfies the first and second laws of thermodynamics, viz.,  energy conservation and entropy production as dynamical variables.  Because thermodynamics is usually applied to continuum systems it is not surprising that early researchers like Eckart  \cite{eckart1} developed  thermodynamically consistent  generalizations of the Navier-Stokes equations  (sometimes call the Navier-Stokes-Fourier equations) that entailed  two dynamical  thermodynamic variables, the specific volume and the specific entropy.  Although Eckart also considered  other systems  \cite{eckart2,eckart3}, many generalizations have appeared over the years, e.g.,  the  Cahn-Hilliard system combined with the  Navier-Stokes system for describing multiphase flows \cite{anderson98,anderson2000phase,Guo}.  
 
To see how entropy is a dynamical variable in a statistical mechanics setting, we refer readers to \cite{lim-oh} for a derivation of nonequilibrium thermodynamics from  statistical mechanics as a symplecto-contact reduction on the kinetic theory phase space,  which is the (affine) space of probability distributions on the many-body particle phase space in the spirit of the earlier  derivations  by Jaynes \cite{jaynes:InfoStat,jaynes:Pp} and Mrugula \cite{Mrug:TPS}.  In this derivation, thermodynamic entropy is covariantly derived from Shanon's information entropy which makes it apparent that entropy is indeed independent of other observables.
 
\emph{Metriplictic  dynamics},   being  a general framework for thermodynamically consistent dynamical systems,  has associated observables (dynamical variables) $H$,  a Hamiltonian  corresponding to the energy  satisfying  $\dot{H}=0$,  and $S$,   an entropy   satisfying $\dot{S}\geq 0$.  Thus the system  has  dynamic realizations 
of the first and second laws of thermodynamics.  The axioms of metriplectic dynamics  appeared in \cite{pjm84,pjm84b} while the name first appeared in \cite{pjm86}.  (See  \cite{pjmU24} for discussion and many references including different formalisms in \cite{pjmK82,Kauf,grmela84}.)

The metriplectic formalism covers a broad range of dynamical systems, both finite- and infinite-dimensional.   The multiphase fluid flows of  \cite{anderson98,anderson2000phase} and \cite{Guo} (with minor correction) were show to be special cases of an encompassing multiphase fluid metriplectic theory given  in  \cite{pjmZB24,pjmZ25}.    In the context of kinetic theory, the Vlasov equation with the Laundau collision operator  was shown to be metriplectic in 
\cite{pjm84,pjm86} and a collision operator suitable for gyrokinetics was given in  \cite{pjmS25}.  See \cite{pjmBKM26} for a  general discussion of the gradient flow nature of metriplectic dynamics and  \cite{pjmU24} for many additional examples, finite and infinite.

The classical \emph{space-time phase space}, which is $T^*\R^n \times \R$ carries the canonical
contact form given by $dt - pdq$, the \emph{contactification} of $T^*M$ \cite[Appendix 4]{arnold:mechanics},
  provides a natural ground for   \emph{time-dependent} Hamiltonian dynamics. 
This time-dependent Hamiltonian dynamics combined with the technique of generating functions 
can be put into  a  \emph{canonical formalism} \cite{arnold:mechanics}, which
Arnold \cite[Chapter 9]{arnold:mechanics} calls  an \emph{odd-dimensional} approach to Hamiltonian phase flows.
Given that   entropy is an independent observable  as mentioned above, it is 
also natural to say that  the space-time phase space is  the correct arena  for the entropy dynamics;  hence,   contact manifolds  become  thermodynamic phase spaces. 

On the other hand,  general \emph{contact Hamiltonian} flows on a contact manifold
$(M, \lambda)$ are  born dissipative,  in that they do  not preserve the natural Liouville measure and hence
Poincar\'e recurrence fails. While there are many similarities in their Hamiltonian calculi, 
this dynamical perspective makes a stark difference from the behavior of the classical (symplectic)
Hamiltonian dynamics. This being said, a fundamental question  in contact dynamics
is the question on how one can systematically study this dissipative aspect 
of contact dynamics in terms of the underlying \emph{contact Hamiltonian
formalism}.  (See \cite{grmela:mesoscopic,esen-grmela-pavelka-I,esen-grmela-pavelka-II,lim-oh,do-oh:reduction} in which the relevant questions are also discussed.)
This paper attempts to make the first step towards this question in the context  of the  \emph{metriplectic formalism} on contact manifolds.  

Preliminary material on contact, symplectic, Poisson, and metriplectic manifolds is  given in Sec.~\ref{sec:intro}, while Sec.~\ref{sec:flows} describes flows on these geometrical structures. Section \ref{sec:comp} compares contact Hamiltonian systems and metriplectic systems. Section \ref{sec:further} discusses a  coordinate-free construction of metriplectic geometry on a  one-jet bundle of a Riemannian manifold.  Section \ref{sec:duff} show how the Duffing equation can be `derived' in either the contact Hamiltonian formalism or the metriplectic formalism, {and some ramifications of the formalisms are explored}.  Finally, in Sec.~\ref{sec:conclusion} we offer some concluding remarks. 

\section{Preliminaries}
\label{sec:prelims}

We consider structures on a  smooth real manifold $\calz$ of dimension $n$, with a point of $\calz$ denoted by $P$. The set of vector fields {on a manifold $\calz$, which is the space $\Gamma(T\calz)$ of sections
of the tangent bundle $T\calz$,  is denoted by $\mathfrak{X}(\calz))$, 
and differential forms on
 $\calz$ by $\Omega^k(\calz) = \Gamma(\La^k(\calz)$}, where $\Omega^0(\calz)$ is the set of smooth functions $\calz\mapsto \R$, $\Omega^1(\calz) = \Gamma(T^*\calz)$ the set of 1-forms, etc.  We denote the exterior derivative by $\bfd$ and wedge product by $\wedge$. We also use Einstein's summation convention throughout the paper.
\medskip

\noindent(i) The manifold $\calz$ is a {\it symplectic manifold} if it is equipped with  a closed nondegenerate 2-form $\om\in  \Omega^2(\calz)$.  For vectors $X,Y\in T_P\calz$ we have  $\om(X,Y)\in\R$.  Because $\om$ is nondegenerate we have the {\it Poisson bivector} defined by $\{f,g\}=  J(\bfd f,\bfd g)$ for $f,g\in\Omega^0(\calz)$,  and $\bfd f,\bfd g\in\Omega^1(\calz)$.  In coordinates $J$ is determined by the {\it Poisson tensor}. As usual,  symplectic manifolds are denoted by the pair $(\calz,\om)$ and the interior product (contraction) by  $\lrcorner$\,.

In {\it Darboux coordinates} the Poisson tensor has the form
\bq
J_c=
\begin{bmatrix}
    0_N       &\  I_N  \\
-I_N     & \ 0_N
\end{bmatrix}\,.
 \eq
where $n=2N$,  $I_N$ is the $N\times N$ identity and $O_N$ is an $N\times N$ block of zeros.  In these coordinates, say $z=(z^1,z^2,\dots, z^{2N})$, there is a split into the canonical coordinates $z=(q,p)$ and the Poisson bracket is given by
\bq
\{f,g\}= J_c^{ij} \frac{\p f}{\p z^i} \frac{\p g}{\p z^i} =\frac{\p f}{\p q^a} \frac{\p g}{\p p_a} -\frac{\p g}{\p q^a} \frac{\p f}{\p p_a} 
\eq
where $i,j=1,2,\dots, 2N$,  $a=1, 2,\dots, N$, and repeated sum notation is assumed.

\bigskip
 
\noindent(ii) The manifold $\calz$ is a {\it Poisson manifold} if it is equipped with a bracket $\{\,,\, \}\colon \Omega^0(\calz)\times \Omega^0(\calz)\rightarrow \Omega^0(\calz)$ producing a Lie algebra realization with the addition of a Leibniz rule (a Poisson algebra). Being a Lie algebra realization means it is antisymmertric, bilinear (here over the field $\R$), and satisfies the Jacobi identity.  {When degeneracy exists, there exist Casimir invariants}  $C$ that satisfy $\{f,C\}\equiv 0$ for all $f\in\Omega^0(\calz)$.  Poisson manifolds are denoted  by the pair $(\calz,\{\,,\,\})$.

In {\it Darboux coordinates} the Poisson tensor has the form
\bq
J=
\begin{bmatrix}
    0_N       &\  I_N & 0_{N\times r}\\
-I_N     & \ 0_N& 0_{N\times r}\\
  0_{r\times N}&  \   0_{r\times N}    & \ 0_r\\
\end{bmatrix}\,.
 \eq
where $J$ is an $n\times n$ matrix with $n=2N+r$ and,  as before,   $I_N$ is the $N\times N$ identity and $O_N$ is an $N\times N$ block of zeros, while $0_{r\times N}$ is a rectangular matrix of zeros. In these coordinates it is obvious that there are $r$ independent (nonparallel gradients) Casimir invariants.

\bigskip
 
\noindent(iii) The manifold $\calz$ is a {\it contact manifold} if it  has odd dimension, say $n=2N+1$, and is equipped with a smooth maximally nonintegrable hyperplane field (does not define a distribution) $\xi \subset T\calz$.  This means that locally $\xi =\mathrm{ker}\, \al$, where $\al$ is a 1-form that satisfies $\al \wedge (\bfd\al)^n \neq 0$.   The quantity $\xi$  is called a contact structure and the quantity  $\al$   a contact 1-form that locally defines $\xi$.  Contact manifolds  are denoted by the pair $(\calz,\xi)$.

In {\it Darboux coordinates} the 1-form $\al$ takes the form
\bq
\al= d z- pd q  \qquad \rightarrow\qquad \bfd\al = dq \wedge d p\,.
\eq
The Reeb vector field, say $R$, which serves as a sort of analog in contact geometry to a Hamiltonian vector field, is given by
\bq
R \,\lrcorner \, \al =1\andqq R\,  \lrcorner \, 
\bfd\al =0 \,.
\eq

\bigskip
 
\noindent(iv) The manifold $\calz$ is a {\it metriplectic manifold} \cite{pjm84,pjm84b}
 if it is a Poisson manifold and possess an additional multilinear 4-bracket
\[
(\,,\, ;\, ,\, )\colon \Omega^0(\calz)\times \Omega^0(\calz)\times \Omega^0(\calz)\times \Omega^0(\calz)\rightarrow \Omega^0(\calz)
\]
where  for any  $f,k,g,n\in\Omega^0(\calz)$ the following    properties hold:

\noindent (a) the algebraic identities/symmetries
\bal
&(f,k;g,n)=-(k,f;g,n)
\nn\\
&(f,k;g,n)=-(f,k;n,g)
\nn
\\
&(f,k;g,n)=(g,n;f,k)
\label{Basym3}
\eal
and
\bal
&(f,k;g,n) + (f,g;n,k)+ (f,n;k,g)=0\,, 
\label{Basym4}
\eal
which are not axiomatically minimal, and 

\noindent (b) derivation in all arguments, e.g., 
\bq
(fh,k;g,n)= f(h,k;g,n)+(f,k;g,n)h
\label{leibniz}
\eq
where $fh$ denotes pointwise multiplication. 
 One way to obtain a metriplectic 4-bracket is via a covariant version of the Kulkarni-Nomizu (K-N) product.
 (See \cite{kulkarni,nomizu}.)
 Given two symmetric bivector fields, say $\sigma$ and $ \mu $, operating on 1-forms $\bfd f, \bfd k$ and $\bfd g,\bfd n$, the K-N product is defined by 
\bal
 \sigma \KN \mu \,(\bfd f, \bfd k,\bfd g,\bfd n)&=  \sigma(\bfd f,\bfd g) \, \mu (\bfd k,\bfd n) 
\nonumber\\
 &- \sigma(\bfd f,\bfd n) \, \mu (\bfd k,\bfd g)
 \nonumber\\
&+  \mu (\bfd f,\bfd g)\,  \sigma(\bfd k,\bfd n)
\nonumber\\
 &- \mu (\bfd f,\bfd n)\,  \sigma(\bfd k,\bfd g) \,.
\label{KNfinite}
\eal
Thus,  the following produces a 4-bracket:
\bq
(f,k;g,n)= \sigma \KN  \mu (\bfd f, \bfd k,\bfd g,\bfd n)\,.
\eq

In a coordinate patch the metriplectic 4-bracket can be expressed in index form as
\bq
(f,k;g,n)=R^{ijkl}(z) \frac{\p f}{\p z^i} \frac{\p k}{\p z^j} \frac{\p g}{\p z^k} \frac{\p n}{\p z^l}\,.
\label{4Bktcoords}
\eq
Here and henceforth $i,j,k,l=1,2,\dots, n$

Any manifold that is at once both a Poisson manifold and a Riemannian manifold or a Poisson manifold with a connection, has metriplectic structure.  However, metriplectic manifolds need not be Riemannian. 


\section{Hamiltonian, contact Hamiltonian, and metriplectic Systems}
\label{sec:flows}

\subsection{Hamilton's equation on Poisson manifolds}

Given a Hamiltonian $H\in\Omega^0(\calz)$,  canonical (symplectic) Hamiltonian vector fields $X_H$ are given by 
\bq
X_H \,  \lrcorner \, \om = d H\,.
\eq
In Darboux coordinates this leads to the {\it canonical Hamiltonian system} of differential equations:
\bq
\dot{q}^a=\frac{\p H}{\p p_a} \andq \dot{p}_a= -\frac{\p H}{\p q^a}, \quad
 a=1,2\dots N\,,
\eq
which follows because $\om$ is nondegenerate.

A Hamiltonian vector field on general Poisson manifolds 
is determined in terms of a Hamiltonian $H\in\Omega^0(\calz)$ via the bracket derivation
\bq
\calv_{NCH}=\{\ \cdot\ ,H\}=  J(\  \cdot\  ,\bfd H)\,.
\eq
(See, e.g.,  \cite{alan:Poisson} and  \cite{pjm98}.)


Thus, in a coordinate patch a  {\it  Hamiltonian system} on a Poisson manifold
is governed by the differential equations
\bq
\dot{z^i} = \{z^i,H\}=J^{ij} \frac{\p H}{\p z^j}\,, \qquad i,j=1,2,\dots, n\,
\eq
in terms of the  Poisson bivector field 
$$
J = J^{ij}\frac{\p}{\p z^i} \wedge \frac{\p}{\p z^j}\,.
$$

In (Poisson) Darboux coordinates \cite{alan:Poisson}, say $(q,p,C)$,  this gives the  {\it noncanonical Hamiltonian system} 
of differential equations:
\bal
\dot{q}^a&=\frac{\p H}{\p p_a}\,,\quad  \dot{p}_a= -\frac{\p H}{\p q^a},  \quad a=1,2,\dots, N
\nn\\
\dot{C}^r&= 0 \,, \qquad  r=1,2,\dots, n-2N. 
\eal

\subsection{Hamilton's equation on contact manifolds}

Let $(M,\al)$ be a contact manifold of dimension $2n+1$.
A contact Hamiltonian vector field is defined by
\bq
\begin{cases}
X_H\, \lrcorner \,  \al = -H   \\
X_H\, \lrcorner \, d\al  = dH-R_\al[H] \al\,,
 \end{cases}
\eq
where $R_\al$ denotes the Reeb vector field and $R_\al[H]$ is the directional derivative action on the 0-form $H$.  
(See \cite{MdL-MLV,oh-yso:index} for discussion -- we follow their  sign conventions.)
In Darboux coordinates $(q,p,z)$   {\it contact Hamiltonian systems} are  given by 
 \bal
 \dot{q}^i&=\frac{\p H}{\p p_i}\,, 
\qquad  \dot{p}_i=-\frac{\p H}{\p q^i} - p_i \frac{\p H}{\p z} \,, 
\\
 \dot{z}&=- H +  p_i\frac{\p H}{\p p_i}\,, \qquad i =1,2,\dots, n\,, 
  \label{cham}
 \eal
 which displays both Hamiltonian and dissipative parts.  {Consequently,  the energy evolution is }
 \bq
 \dot{H}= \frac{\p H}{\p q^i} \dot{q}^i+ \frac{\p H}{\p p_i} \dot{p}_i+ \frac{\p H}{\p z} \dot{z}
 = -H \frac{\p H}{\p z}  \,,
 \eq
 which does not vanish in general, but may if $H=0$ or $\p H/\p z=0$. 
 
\subsection{Metriplectic dynamical systems and thermodynamic consistency}
 
As already noted, thermodynamic consistency means $\dot{H}=0$ and $\dot{S}\geq 0$, which are dynamic realizations  of the first and second laws of thermodynamics. (See \cite{pjm84,pjm84b,pjmU24} for relevant discussions.)  Neither of Hamilton's equations   is   suitable to
accommodate this thermodynamic consistency which  motivated the first-named author to 
introduce \emph{metriplectic dynamical systems} which is now in order.
 
A metriplectic  vector field has both  Hamiltonian and dissipative parts,  generated as follows:
\bq
\calv_{MP}=\{\ \cdot\ ,H\} + \{\ \cdot\ ,H; S, H\}
\eq
where $S,H\in \Omega^0(\calz)$ are an  entropy function and Hamiltonian, respectively.  

In a coordinate patch, a {\it metriplectic dynamical system} is given in terms of a metriplectic  vector field  according to
\bal
\dot{z}^i&=  \{z^i,H\}+ \{z^i  \cdot\ ,H; S, H\}
\nn\\
&=J^{ij} \frac{\p H}{\p z^j} + 
R^{ijkl} \frac{\p H}{\p z^j}\frac{\p S}{\p z^k}\frac{\p H}{\p z^l}
\,, 
\eal
 where $i,j,k,l=1,2,\dots, n$.  Thermodynamic consistency follows if $S$ is a Casimir invariant, i.e., $\{f,S\}\equiv 0$ for all $f\in\Omega^0(\calz)$ and  $\{f,k; g, n\}$ is a 4-bracket, e.g., generated by a K-N product.  In the next section we will display a particular  form for the 4-bracket that facilitates comparison with contact Hamiltonian systems. 

\section{Contact Hamiltonian systems versus metriplectic dynamical systems}
\label{sec:comp}

To motivate our introduction of metriplectic dynamical systems on the one-jet bundle $J^1M = T^*M \times \R$, 
which we may regard as  the \emph{space-time phase space}. The upshot is that  this bundle  carries the canonical
contact form given by $dt - pdq$.  For the simplicity of presentation, which however contains the essence of the comparison, we start with
the simplest contact Hamiltonian and metriplectic systems on $\R^2\times \R = \R^3$.

First consider a {\it 3D contact Hamiltonian system}.  For the coordinates $(q,p,z)\in \R^3$ Eqs.~\eqref{cham} reduce to the following three ODEs:
 \bal
 \dot{q}&=\frac{\p H}{\p p}\,,\quad
 \dot{p}=-\frac{\p H}{\p q} - p \frac{\p H}{\p z} 
\nn\\
 \dot{z}&=- H + p\frac{\p H}{\p p}\,.
 \label{3Dch}
 \eal
and as above, 
\bq
\dot H = -H \frac{\p H}{\p z}  \neq 0
\label{CHenergy}
\eq
Thus energy is not conserved unless $H$ is independent of $z$, {$\p H/\p z=0$,} or one begins  on the $H=0$ level set.  So, in general contact Hamiltonian systems are {\bf not} thermodynamically consistent. 

\medskip

We build our {\it 3D metriplectic system}  atop a  trivial Poisson manifold,  which is a stack of symplectic planes $T^*\calz\times\R$ (the first jet bundle).  In 3D we have  $\calz= T^*\R\times \R$, which has the following Poisson bivector $J$ in Darboux coordinates $(q,p,z)\in \R^3$: 
\bq
J=
\begin{bmatrix}
    0       &\  1 &\  0\ \\
-1     & \ 0 &\ 0\  \\
0    &\ 0 &  \ 0\ 
\end{bmatrix}\,.
 \eq
That $J$ satisfies the Jacobi identity is trivial. For this construction, these coordinates are in fact global.  Thus the Hamiltonian part of our 3D metriplectic system is generated by this Poisson bracket as 
\bq
\dot q = \{q,H\}\,,\ \   \dot p  = \{p,H\}\,, \ \   \mathrm{and}\ \  \dot z  = \{z,H\}=0
\nn
\eq
or
\bq
\begin{bmatrix}
\dot q\\
\dot p\\
\dot z
\end{bmatrix}
=\begin{bmatrix}
    0       &\  1 &\  0\ \\
-1     & \ 0 &\ 0\  \\
0    &\ 0 &  \ 0\ 
\end{bmatrix}
\begin{bmatrix}
\p H/\p q\\
\p H/\p p\\
\p H/\p z
\end{bmatrix}\,.
\eq
Here $z$ is the Casimir invariant that plays the role of entropy, i.e.,  $\{f,z\}=0$ for all $f$. Thus we are on the road to thermodynamic consistency  since evidently $\dot{H}\equiv 0$.  
 
The full metriplectic dynamics in the standard Darboux coordinates of $\R^2 \times \R$ 
{will be}  given by 
\bq
\dot{z}^i= \{z^i,H\} + (z^i, H;S,H)\,,\qquad i=1,\, 2,\, 3, 
\eq
yielding 
\bal
\dot q &= \frac{\p H}{\p p}\,,\qquad 
\dot p = - \frac{\p H}{\p q} - \,\frac{\p H}{\p p}  \frac{\p H}{\p z}  \,,
\nn\\
\dot z &=    \left(\frac{\p H}{\p p}\right)^2.
\label{fmp}
\eal
Here we have equated the entropy $S$ with the coordinate $z$ and  particular choices were made for $\mu$ and $\si$, which will be addressed in Sec.~\ref{sec:further}.  (See \eqref{eq:mu-sigma}.)

It remains to compare \eqref{fmp} with \eqref{3Dch}. The first observation to make is that for 
 so-called {\it natural } Hamiltonians where $H=p^2/2 + V(q)$, $\p H/\p p= p$.  So,  \eqref{fmp} gives 
$\dot{z}= p^2$,  while  \eqref{3Dch} yields the Legendre transform of $H$, i.e. the Lagrangian 
\bq
 L= p \frac{\p H}{\p p} - H= p^2/2 -V(q)\,,
\eq
 which is not sign definite.  However, if  $H$ is the kinetic energy which is
 Euler homogeneous of degree 2 in $p$, then  $\dot{z}= -H + p {\p H}/{\p p} =  H = (\p H/\p p)^2/2$ and the above two systems can be seen to coincide within a factor of 2.  If in addition the Hamiltonian is linear in $z$ so that $\p H/\p z=1$ then the equations for $\dot{p}$ of the two systems coincides.  This demonstrates that  for some Hamiltonians and entropies, contact Hamiltonian systems can be thermodynamically consistent.  This can  generalizes to the kinetic energy 
 on the cotangent bundle $T^*N$ equipped with a Riemannian metric of arbitrary dimension.  
 
 By altering $\mu$,  $\si$, $S$, and $H$ one can obtain other corresponding systems with various properties.

\section{Metriplectic systems on the one-jet bundle}
\label{sec:further}

Let $(N,g)$ be any Riemannian manifold and consider its one-jet bundle $J^1N = T^*N \times \R$
equipped with the canonical contact form
$$
\alpha = dz - p\, dq = dz - \sum_{i=1}^n p_i dq^i.
$$
We denote by $\omega_0$ the standard symplectic form $\omega_0$ on
$T^*N$ and pull it back to $\pi^*\omega_0$ on $J^1N$ which is a closed two-form.
Thus, it  defines a Poisson bracket on $J^1N = T^*N \times \R$ with coordinate function $z$
as a Casimir. We denote this bracket by $\{\cdot, \cdot\}$, which governs the Hamiltonian 
part of the dynamical system.

To naturally incorporate the dissipative aspect of certain dynamical systems,
we  introduce some natural metriplectic 4-brackets associated to the contact structure.

For this purpose, we use the simplest  Kulkarni-Nomizu product to construct the metriplectic 4-bracket: 
\bq\label{eq:mu-sigma}
\mu(df,dg)= \frac{\p f}{\p z} \frac{\p g}{\p z}   \andq \si 
= \left \langle \frac{\p f}{\p p}, \frac{\p g}{\p p} \right\rangle  \,,
\eq
where $\left\langle \frac{\p f}{\p p}, \frac{\p g}{\p p} \right\rangle$ denotes the dot product of 
 the \emph{fiber derivatives} $d^\text{\rm fiber}f$ of $f$ and $g$ for the projection $T^*N \to N$. 
By altering $\mu$,  $\si$, $S$, and $H$ one can obtain other corresponding systems.

Here we recall the  precise definition of the fiber derivative that  applies to any vector bundle $E \to B$, 
which we apply to $T^*N \to N$ (or $J^1N = T^*N \times \R$) in the present paper.

\begin{definition} Let $E \to B$ be a vector bundle.
For each $e_b \in T_b B$, we consider
\bq\label{eq:fiber-derivative}
\frac{\del f}{\del p}(e_b)(\eta_p): = \frac{d}{ds}\Big|_{s = 0} f(e_b + s \eta_b).
\eq
This defines a section element of $\Omega^1(B)$ 
given by the 
$$
b \mapsto d^{\text{\rm fiber}}_b f \in E^*_b,,
$$
which we denote by $d^{\text{\rm fiber}} f$.
\end{definition}
Then the precise expression of our  heuristic notation $\left \langle \frac{\p f}{\p p}, \frac{\p g}{\p p} \right\rangle$ is 
$$
\left \langle \frac{\p f}{\p p}, \frac{\p g}{\p p} \right\rangle: = g^\#(d^{\text{\rm fiber}} f, d^{\text{\rm fiber}} g).
$$
In canonical coordinates $(q^1, \cdots, q^n, p_1, \cdots, p_n)$ of $T^*N \times \R$, we have
$$
g^\#(d^\text{\rm fiber}f, d^\text{\rm fiber}g) =  g^{ij}  \frac{\p f}{\p p_i} \frac{\p g}{\p p_j}.
$$
\begin{remark} Having given the above  covariant definition, we will keep the  `heuristic'  coordinate expressions below, so as  not to make formulas too
abstract.
\end{remark}

We then take the K-N product $\sigma \KN \mu$. These yield the 4-bracket
\bal
(f,k;g,n) & =
\\
& \hspace{-.5cm}  \left \langle \frac{\p f}{\p p}, \frac{\p g}{\p p} \right\rangle    \frac{\p k}{\p z} \frac{\p n}{\p z}  
-  \left\langle \frac{\p f}{\p p}, \frac{\p n}{\p p} \right\rangle  \frac{\p k}{\p z} \frac{\p g}{\p z}\nn \\
& \hspace{-.9cm}   +   \left \langle \frac{\p k}{\p p}, \frac{\p n}{\p p}  \right\rangle  \frac{\p f}{\p z} \frac{\p g}{\p z}  
-  \left\langle \frac{\p k}{\p p}, \frac{\p g}{\p p} \right\rangle   \frac{\p f}{\p z} \frac{\p n}{\p z} 
\nn
\eal
Recall that  $S=z$ is a Casimir. Then a straightforward calculation leads to the following.

The full metriplectic dynamics in the standard Darboux coordinates of $J^1N$ $(q^1,\cdots,q^n,p_1, \cdots p_n,z)$
 is  given by 
\bq
\dot{z^i}= \{z^i,H\} + (z^i, H;S,H)\,,\qquad i=1,2,3, \cdots, \, n,, 
\eq
yielding 
\bal
\dot q^i &= \frac{\p H}{\p p_i}\,,\qquad 
\dot p_k = - \frac{\p H}{\p q^k} - g_{kj}\,\frac{\p H}{\p p_j}  \frac{\p H}{\p z}  \,,
\nn\\
\dot z &=      \left\|\frac{\p H}{\p p} \right\|_g^2
\label{eq:fmp-general}
\eal

\begin{theorem}
For any choice of Hamiltonian $H = H(q,p,z)$, we have
\bq
\dot H= (H,H;S,H)=0
\eq
while 
\bq
\dot S= (S,H;S,H)=   \left\|\frac{\p H}{\p p} \right\|_g^2 \geq 0
\label{dotS}
\eq
In particular, the system  Equations \eqref{fmp} are indeed thermodynamically consistent. 
\end{theorem}
\begin{proof} Since both equations are coordinate independent, we
consider the standard Darboux coordinates of $J^1N$ $(q^1,\cdots,q^n,p_1, \cdots p_n,z)$
noting that $z$ is a global function indepependent of this choice.

Clearly $\dot{z}\geq 0$ and by the metriplectic symmetries, we derive
\bq
\dot{H}=   \frac{\p H}{\p p}\left( - \frac{\p H}{\p p}  \frac{\p H}{\p z}  
\right)  +  \frac{\p H}{\p z}   \left(\frac{\p H}{\p p} \right)^2 = 0\,.
\eq
\end{proof}

\section{An example: Duffing's equation}
\label{sec:duff}

The well-known Duffing equation \cite{duffing,thompson-steward}, being an equation with both damping (friction) and drive,   provides an interesting example that shows how nonautonomous systems can be at once contact Hamiltonian and metriplectic.  We will derive it in this section   as a contact Hamiltonian system and as a metriplectic system.  We will see in both cases that the usual Duffing equation is a subsystem of  three-dimensional systems  that contain a thermodynamic component. 

The most general form of the  Duffing equation is given by  the following nonlinear ODE with external periodic forcing:
\bq
{\ddot {q}}+\delta {\dot {q}}+\alpha q+\beta q^{3}= \gamma \sin(\omega t + \phi),
\label{duffing}
\eq
with parameters $\delta, \, \alpha, \, \beta, \, \gamma, \, \omega, \, \phi$.

\subsection{Derivation of Duffing's equation   {and energy flow} }
 
All that is needed to obtain  the contact Hamiltonian form is to specify a contact Hamiltonian.  We assume
\bq
\label{eq:Hduffing}
H = \frac{p^2}{2} + \frac{\alpha q^2}{2} + \frac{\beta q^4}{4} - \gamma q \sin(\omega t + \phi) + \delta z.
\eq
Observe, if we drop  the last term  of \eqref{eq:Hduffing}, $\de z$,  we have an ordinary Hamiltonian for the Duffing equation without the damping term, $\de \dot{q}$,  of \eqref{duffing}.  However, if we substitute \eqref{eq:Hduffing} into \eqref{3Dch}
we obtain the  contact Hamiltonian system  given by
\bal
\dot q &= p
\label{eq:contact-duffing-q}\\
\dot p &= - \alpha q - \beta q^3 + \gamma \sin(\omega t + \phi) - \delta p
\label{eq:contact-duffing-p}\\
\dot z &= \frac{p^2}{2} - \delta z - \frac{\alpha q^2}{2} - \frac{\beta q^4}{4} + \gamma q \sin(\omega t + \phi).
\label{eq:contact-duffing-z}
\eal
Upon differentiating \eqref{eq:contact-duffing-q}  with respect to time and using \eqref{eq:contact-duffing-p},  we obtain precisely \eqref{duffing}.  Thus, given a solution of Duffing's equation, the right hand side of \eqref{eq:contact-duffing-z} becomes a function of $z$ and $t$, and it can be solved after the fact. 

Because this  system is a nonautonomous contact Hamiltonian system,  \eqref{CHenergy} becomes
\bq
\dot{H} = -H \frac{\p H}{\p z} + \frac{\p H}{\p t}= -\de H - \gamma \om q \cos(\omega t + \phi)\,;
\eq
upon dropping the driving term we see that the Hamiltonian decays according to
\bq
H=H_0\, e^{-\de t}\,, 
\eq
and the entropy equation \eqref{eq:contact-duffing-z} depends on all the variables $(q,p,z,t)$. 

On the other hand, the associated metriplectic system becomes the following upon substitution of the $H$ of \eqref{eq:Hduffing} into \eqref{fmp}:
\bal
\dot q &= p
\label{eq:metriplectic-duffing-q}\\
\dot p &= - \alpha q - \beta q^3 + \gamma \sin(\omega t + \phi) - \delta\,  p
\label{eq:metriplectic-duffing-p}\\
\dot z &=  p^2
\label{eq:metriplectic-duffing-z}
\eal
 {Noting that the first two equations of this system coincide with  \eqref{eq:contact-duffing-q}  and  \eqref{eq:contact-duffing-p},  we again obtain precisely Duffing's equation.}  And, again \eqref{eq:metriplectic-duffing-z} can be solved after the fact, but now it only has dependence on $p$;  it does not have explicit time dependence arising from the  nonautonomous term of $H$ nor any dependence on $q$.  

{Upon dropping the driving term, the system satisfies $\dot{H}=0$, as it must since it is autonomous metriplectic.  On the other hand for the case with driving when it is  nonautonomous metriplectic, }
\bq
\dot{H}= \frac{\p H}{\p t} =  - \gamma \om q \cos(\omega t + \phi)\,.
\eq
However, it  is interesting that we still have
\[
\dot{z}\geq 0\,,
\]
even when  the system is driven!
  
In this  metriplectic case the addition of the $\de z$ term to the Hamiltonian of \eqref{eq:Hduffing} makes physical sense as an  internal energy term.   This parallels the case of the metriplectic Navier-Stokes-Fourier system 
(see e.g.\  \cite{pjm84, pjmU24}) where viscous dissipation of energy is compensated for by the production of heat in an entropy equation.   This is an example of a kind of metriplectic thermodynamic completion, where a thermodynamically inconsistent system is made consistent via an entropy equation that is consistent with the second law.

\section{Conclusion}
\label{sec:conclusion}

We have shown that both contact Hamiltonian and metriplectic dynamical systems naturally occur on the one-jet bundle.  
We have compared and contrasted  the resulting flows, and have provided the Duffing equation as an example,
{which can be derived either as a contact Hamiltonian system or as a metriplectic
dynamical system}.  This example shows how damping/friction can appear in both settings. 

Both contact Hamiltonian and metriplectic dynamical systems  have {analogous constructions} for infinite-dimensional systems, field theories (PDEs).  Lifting the results of this setting to such systems is a natural next step.  As noted in Sec.~\ref{sec:intro}, understanding how both systems emerge from underlying large dimensional  Hamiltonian systems is a goal.


\section*{Acknowledgment}
\noindent   PJM received support from  U.S. Dept.\ of Energy Contract \# DE-FG05-80ET-53088 and
YGO's research is supported by the IBS project \# IBS-R003-D1. Both authors were 
supported by the 2025 SLMath program on kinetic theory.







%
\bibliographystyle{elsarticle-num} 



\def\cprime{$'$}




\end{document}